\documentclass[11pt,twoside,letter]{article} 

\usepackage{amssymb}
\usepackage{graphicx,rotating}
\usepackage{amscd}
\usepackage{color}
\usepackage{float}
\usepackage{graphics}
\usepackage{amsmath}
\usepackage{amssymb}
\usepackage{mathrsfs}
\usepackage{amsthm}
\usepackage{amsfonts}
\usepackage{latexsym}
\usepackage{epsf}
\usepackage[]{hyperref}
\usepackage{fancyhdr}
\usepackage{mathptmx}
\usepackage{longtable}
\usepackage{array}
\usepackage{csquotes}
\usepackage{verbatim}
\usepackage{enumerate}
\usepackage[margin=1.5in]{geometry}
\usepackage{rotating}
\usepackage{caption}
\usepackage{mdframed}
\usepackage{colortbl}

\newcolumntype{P}[1]{>{\centering\arraybackslash}p{#1}}
\newmdtheoremenv{theo}{Theorem}

\makeatletter
\newenvironment{proof*}[1][\proofname]{\par
  \pushQED{\qed}%
  \normalfont \partopsep=\z@skip \topsep=\z@skip
  \trivlist
  \item[\hskip\labelsep
        \itshape
    #1\@addpunct{.}]\ignorespaces
}{%
  \popQED\endtrivlist\@endpefalse
}
\makeatother

\date{}

\begin{document}

\centerline{\bf Central Luzon State University}

\centerline{\bf Undergraduate Research in Mathematics}

\centerline{Academic Year 2021-2022} 

\centerline{} 

\centerline {\Large{\bf On $k$-dprime Divisor Function Graph}} 
\vspace{2mm}










\newtheorem{theorem}{Theorem}[section]
\newtheorem{lemma}[theorem]{Lemma}
\newtheorem{corollary}[theorem]{Corollary}
\newtheorem{proposition}[theorem]{Proposition}
\newtheorem{definition}[theorem]{Definition} 
\newtheorem{example}[theorem]{Example}
\newtheorem{remark}[theorem]{Remark}
\newtheorem{illustration}[theorem]{Illustration}


\begin{abstract} 
Let $p$ and $q$ be distinct primes. The \textit{semiprime divisor function graph} denoted by $G_{D(pq)}$, is the graph with vertex set $V(G_{D(pq)})=\{1,p,q,pq\}$ and edge set $E(G_{D(pq)})=\{\{1,p\}, \{1,q\},\{1,pq\},\{p,pq\},\{q,pq\}\}$. The semiprime divisor function graph is a special type of divisor function graph $G_{D(n)}$ in which $n=pq$. Recently, the energy and some indices of semiprime divisor function graph have been determined. In this paper, we introduce a natural extension to the semiprime divisor function graph which we call the \textit{$k$-dprime divisor function graph}. Moreover, we present results on some distance-based and degree-based topological indices of $k$-dprime divisor function graph.  We end the paper by giving some open problems. 
\end{abstract} 
{\small{{\bf Keywords:} $k$-dprime divisor function graph, semiprime divisor function graph, divisor function graph, topological indices\\
 {\bf AMS Classification Numbers:} 05C12, 05C50, 05C85}}
\hrule
\vspace{4mm}

\begin{flushleft}
{\footnotesize{{\bf Author Information:}
\vspace{3mm}

\underline{John Rafael M. Antalan}

Assistant Professor

Department of Mathematics and Physics, College of Science, Central Luzon State University (3120), Science City of Mu\~{n}oz, Nueva Ecija, Philippines.

e-mail: jrantalan@clsu.edu.ph 
\bigskip

\underline{Jerwin G. De Leon}

Student

Department of Mathmatics and Physics, College of Science, Central Luzon State University (3120), Science City of Mu\~{n}oz, Nueva Ecija, Philippines. 

e-mail: deleon.jerwin@clsu2.edu.ph 
\bigskip

\underline{Regine P. Dominguez}

Student

Department of Mathmatics and Physics, College of Science, Central Luzon State University (3120), Science City of Mu\~{n}oz, Nueva Ecija, Philippines. 

e-mail: dominguez.regine@clsu2.edu.ph
  
}}
\end{flushleft}

\newpage

\section{Introduction}

One of the developing areas in Graph theory is the notion of using Number theory concepts to define graphs. The said graphs are called \textit{number theoretic based graphs}. One of the most studied number theoretic based graph is the \textit{divisor graph}. Let $S$ be a non-empty subset of $\mathbb{Z}$, a graph $G(V,E)$ is a \textbf{divisor graph} if $V(G)=S$ and $E(G)=\{ij:\mbox{either $i\mid j$ or $j\mid i$ for $i,j\in V(G)$ with $i\neq j$}\}$. The concept of divisor graph was introduced by Singh and Santhosh \cite{Singh} in 2000. Since then, various research studies about divisor graph have been conducted (see \cite{Chart,Frayer,Vinh,Tsao}).
	
	Motivated by the concept of divisor graph, Kannan et al. \cite{Kan} introduced the concept of \textit{divisor function graph} in 2015. Let $n\geq 1$ be an integer, and suppose that $n$ has $r$ divisors $d_1, d_2, \ldots, d_r$, the \textbf{divisor function graph} of $n$ denoted by $G_{D(n)}(V,E)$ is the graph with $V(G_{D(n)})=\{d_1,d_2,\ldots,d_r\}$ and $$E(G_{D(n)})=\{d_id_j:\mbox{either $d_i\mid d_j$ or $d_j\mid d_i$ for $d_i,d_j\in V(G)_{D(n)}$ with $i\neq j$}\}.$$ 
	
	If in the definition of the divisor function graph we have $n=pq$, for distinct primes $p$ and $q$, then $G_{D(n)}$ is called a \textbf{semiprime divisor function graph}. The concept of semiprime divisor function graph was studied recently by Shanmugavelan et al. in \cite{Shan}, and was introduced by Narasimhan et al. \cite{Nara} in 2018. In \cite{Shan},  Shanmugavelan et al. determined the energy and some indices of the semiprime divisor function graph.
	
	Inspired by the work of Shanmugavelan et al., we introduce a natural extension to the semiprime divisor function graph which we call the \textit{$k$-dprime divisor function graph} in this paper. We then determine some distance-based and degree-based topological indices of the $k$-dprime divisor function graph. We also give some problems that the reader might consider as a research study. 

\section{The $k$-dprime Divisor Function Graph}

Unless otherwise stated, we follow the graph theory notations of Bondy and Murty \cite{Bon} and the number theory notations of Burton \cite{Rose}. We now formally define the $k$-dprime divisor function graph. 
	
	\begin{definition}
		Let $n\geq 1$ be an integer such that $n=p_1p_2\ldots p_k$, where each $p_i$ are distinct primes for $i=1,2,\ldots, k$. The graph $G_{D(n)}(V,E)$ with $V(G_{D(n)})=\{u:\mbox{$u\mid n$}\}$ and $$E(G_{D(n)})=\{uv:\mbox{either $u\mid v$ or $v\mid u$ for $u,v\in V(G_{D(n)})$ with $u\neq v$}\}$$
		is called a \textbf{\textit{$\textbf{k}$-dprime divisor function graph}.}
	\end{definition}

    \begin{example}
    The graph of $3$-drpime and $4$-dprime divisor function graph is given in Figure \ref{34dprime}. On the other hand, the graph of $5$-dprime divisor function graph is given in Figure \ref{5dprime}.

	\begin{figure}[!ht]
	\centering
	\includegraphics[width=5cm]{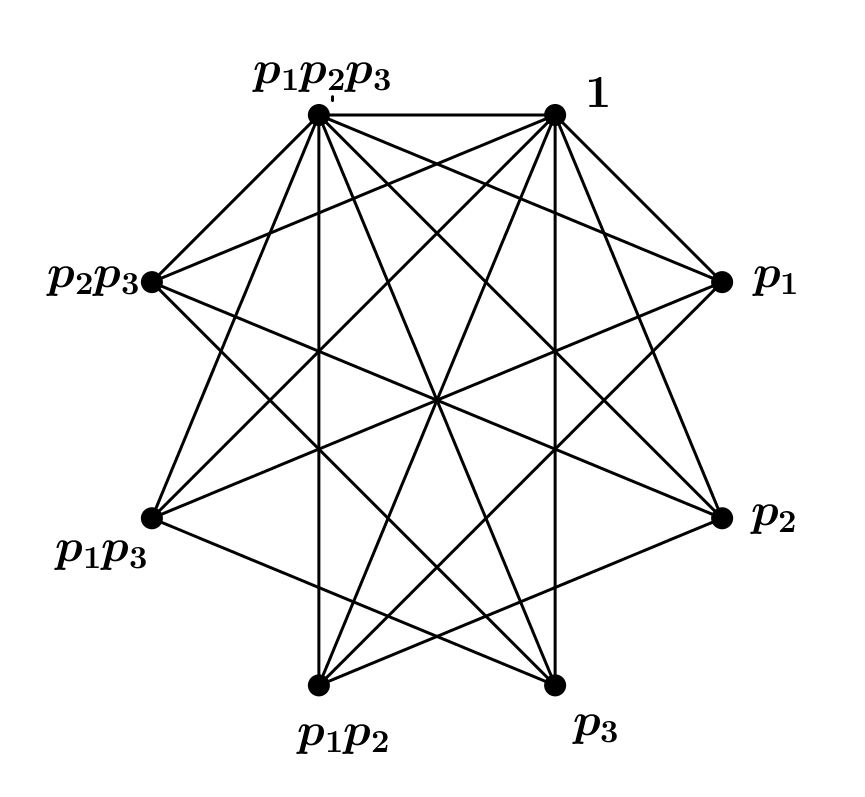}
	\includegraphics[height=5cm]{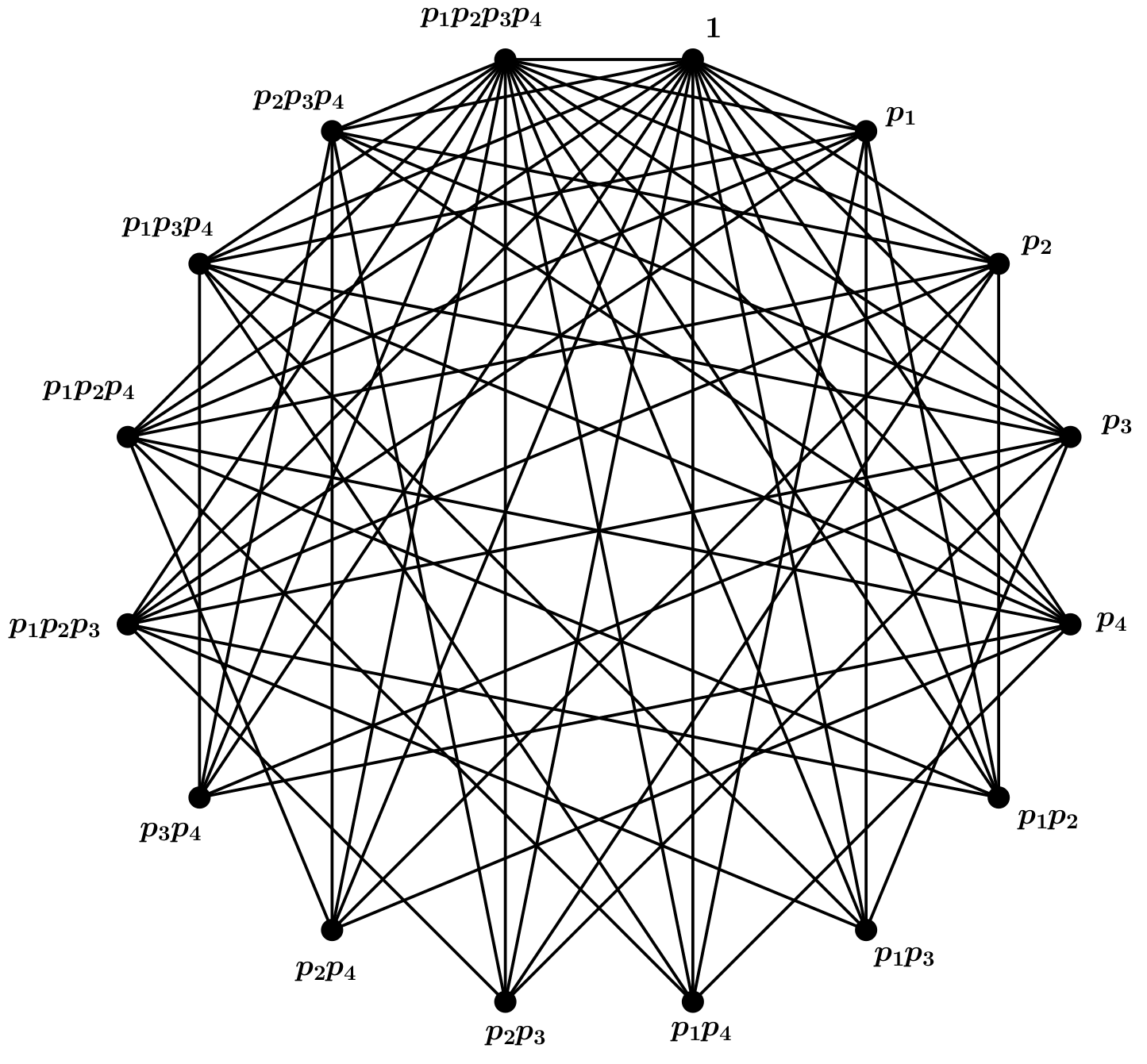}
	\caption{The $3$-dprime and $4$-dprime divisor function graph.}\label{vfig1}
	\label{34dprime}
	\end{figure}
	\begin{figure}[!ht]
	\centering
	\includegraphics[height=8cm]{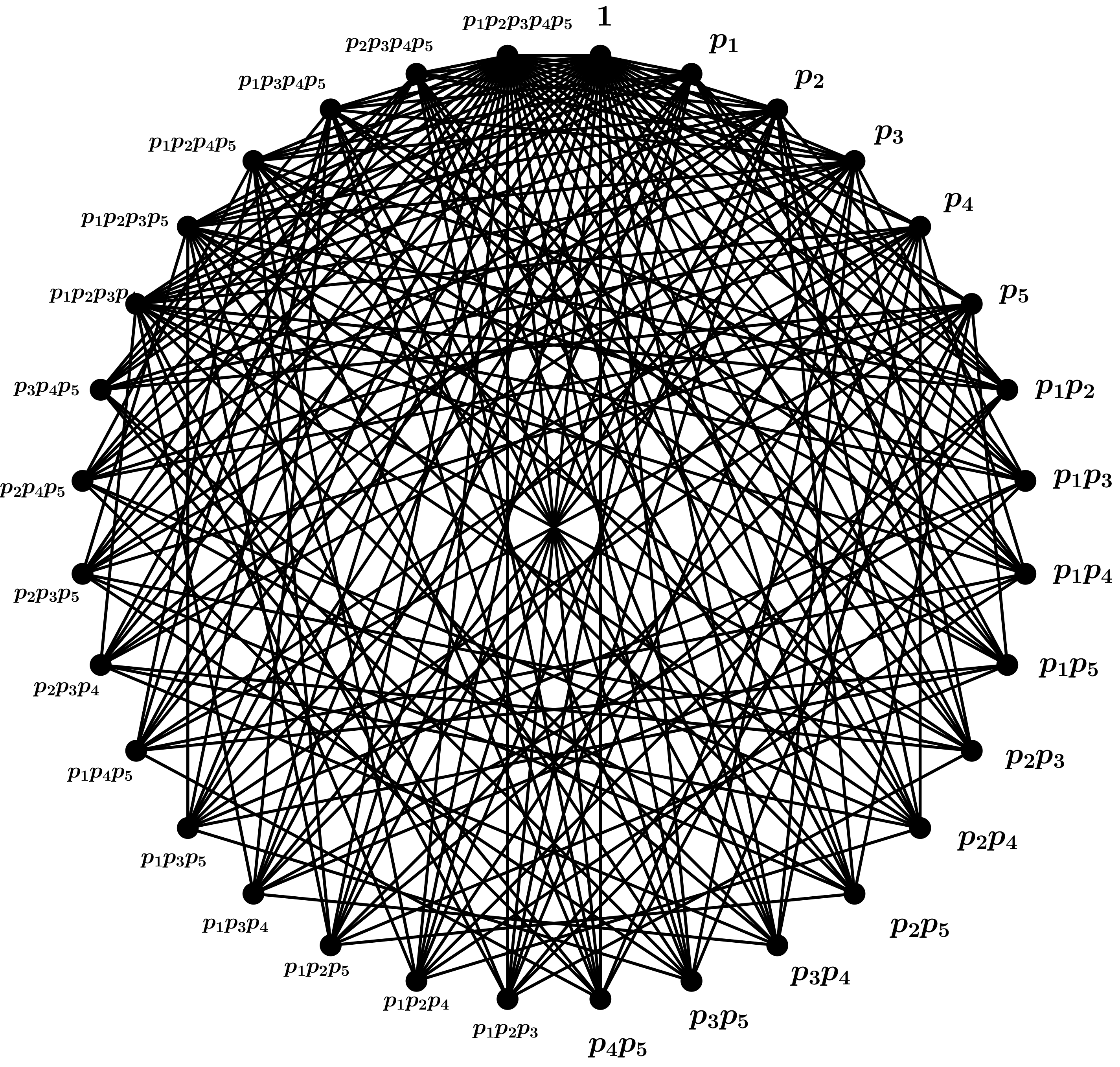}
	\caption{The $5$-dprime divisor function graph.}\label{vfig1}
	\label{5dprime}
	\end{figure}
    
    Observe that the number of vertices in $3$-dprime, $4$-dprime, and $5$-dprime divisor function graph are $8$, $16$, and $32$, respectively. Moreover, the degree sequence of the vertices in $3$-dprime, $4$-dprime, and $5$-dprime divisor function graph are $(7,4,4,4,4,4,4,7)$, $(15,8,8,8,8,6,6,6,6,6,6,8,8,8,8,15)$, and $(31,16,16,16,16,16,10,10,10,10,10,10,10,$\\ $10,10,10,10,10,10,10,10,10,10,10,10,10,16,16,16,16,16,31)$, respectively. Finally, the number of edges in $3$-dprime, $4$-dprime, and $5$-dprime divisor function graph are $19$, $65$, and $211$, respectively.   
    
    \label{sampol1}
    \end{example}

For simplicity, we just denote by $\Gamma_k$ the $k$- dprime divisor function graph $G_{D(n=p_1p_2\ldots p_k)}$. Some of the basic properties of $\Gamma_k$ is given in the next series of results. 

	\begin{theorem}
		The number of vertices in $\Gamma_k$ is $2^k$, that is, the order of $\Gamma_k$ is $2^k$.
		\label{teovernum}
	\end{theorem}
	
	\begin{proof}
		The result follows from the composition of vertices in $\Gamma_k$ and the fact that an integer with canonical representation $p_1^{e_1}p_2^{e_2}\ldots p_r^{e_r}$ has $(e_1+1)\cdot(e_2+1)\cdot\ldots\cdot(e_r+1)$ divisors.  
	\end{proof}
	
	\begin{remark}
	The number of vertices in $3$-dprime, $4$-dprime, and $5$-dprime divisor function graph that were obtained in Example \ref{sampol1} agrees with Theorem \ref{teovernum}.  
	\end{remark}

	\begin{theorem}
		If $v\in V(\Gamma_k)$, then  
		
		\begin{equation*}
		    deg_{\Gamma_k}(v)=
		    \begin{cases}
		      2^k-1 & \mbox{if $v=1,n$}\\
		      2^{\omega(v)}+{2}^{k-\omega(v)}-2 & \mbox{otherwise} 
		    \end{cases}
		\end{equation*}
		
		where $\omega(v)$ is the number of distinct prime divisors of $v$. 
    \label{teoverdeg}
	\end{theorem}

	\begin{proof}
	    
	    To prove the theorem, we will consider 3 cases. 
	    
	    \textbf{Case 1: If $v=1$}. If $v=1$, then there are $2^k$ integers in $V(\Gamma_k)$ (including $v$) that are divisible by $v$. By noting that in $E(\Gamma_k)$ we must have $u\neq v$, we conclude that there are $2^k-1$ vertices that are incident to $v$, by considering the number of integers in $V(\Gamma_k)$ that are divisible by $v$. Next, we consider the number of integers in $V(\Gamma_k)$ that divides $v$. The only integer in $V(\Gamma_k)$ that divides $v$ is $v$ itself. By noting that in $E(\Gamma_k)$ we must have $u\neq v$, we conclude that there is no vertex incident to $v$, by considering the number of integers in $V(\Gamma_k)$ that divides $v$. All in all, we have $2^{k}-1$ edges incident to $v$. Hence, $deg_{\Gamma_k}(v)=2^k-1$.
	    
	    \textbf{Case 2: If $v=n$}. If $v=n$, then there is one integer in $V(\Gamma_k)$ that is divisible by $v$, $v$ itself. By noting that in $E(\Gamma_k)$ we must have $u\neq v$, we conclude that there is no vertex incident to $v$, by considering the number of integers in $V(\Gamma_k)$ that are divisible by $v$. Next, we consider the number of integers in $V(\Gamma_k)$ that divides $v$. There are $2^{k}$ integers in $V(\Gamma_k)$ (including $v$) that divides $v$. By noting that in $E(\Gamma_k)$ we must have $u\neq v$, we conclude that there are $2^k-1$ vertices incident to $v$, by considering the number of integers in $V(\Gamma_k)$ that divides $v$. All in all, we have $2^{k}-1$ edges incident to $v$. Hence, $deg_{\Gamma_k}(v)=2^k-1$. 
	    
	    \textbf{Case 3: If $v\in V(\Gamma_k)-\{1,n\}$}. Let $v\in V(\Gamma_k)-\{1,n\}$ and denote by $\omega(v)$ the number of its distinct prime divisors. Note that since $v\mid n$, and $n=p_1p_2\ldots p_k$, we know that $v$ is a product of $\omega(v)$ distinct primes. 
	    
	    Now, let us first count the number of integers in $V(\Gamma_k)$ that divides $v$. Since $v$ is a product of $\omega(v)$ distinct primes, if we use the fact that an integer with canonical representation $p_1^{e_1}p_2^{e_2}\ldots p_r^{e_r}$ has $(e_1+1)\cdot(e_2+1)\cdot\ldots\cdot(e_r+1)$ divisors, we conclude that there are $2^{\omega(v)}$ integers in $V(\Gamma_k)$ that divides $v$. But in $E(\Gamma_k)$ we must have $u\neq v$, so if we consider the number of integers in $V(\Gamma_k)$ that divides $v$, we have $2^{\omega(v)}-1$ edges that are incident to $v$. Before we proceed, we note that the vertices $u$ that are incident to $v$ in this case satisfies the inequality $\omega(u)<\omega(v)$. This is because (i) $u\leq v$ and (ii) if $u\in V(\Gamma_k)$ such that $u\neq v$, and $\omega(u)=\omega(v)$ then $u\nmid v$.    
	    
	    Next, we count the number of integers in $V(\Gamma_k)$ that are divisible by $v$. Note that if $u\in V(\Gamma_k)$ such that $u\neq v$, and $\omega(u)=\omega(v)$ then $v\nmid u$. So we start the counting for integers $u$ with $\omega(u)>\omega(v)$. By \textbf{counting}, there are $\binom{k-\omega(v)}{1}$ integers $u$ in $V(\Gamma_k)$ with $\omega(u)=\omega(v)+1$. Similarly, by counting, we know that there are $\binom{k-\omega(v)}{2}$ integers $u$ in $V(\Gamma_k)$ with $\omega(u)=\omega(v)+2$. In general, for $j=1,2,\ldots, k-\omega(v)$, a counting technique asserts that there are $\binom{k-\omega(v)}{j}$ integers $u$ in $V(\Gamma_k)$ with $\omega(u)=\omega(v)+j$. All in all we have $\displaystyle \sum_{j=1}^{{k-\omega(v)}}\binom{k-\omega(v)}{j}$ number of integers in $V(\Gamma_k)$ that are divisible by $v$ that are not equal to $v$.     
	    
	    Hence, there are $\displaystyle  2^{\omega(v)}-1+\sum_{j=1}^{{k-\omega(v)}}\binom{k-\omega(v)}{j}$ number of vertices that are incident to $v$. If we use the identity $\displaystyle \sum_{j=0}^n{\binom{n}{j}}=2^n$, we conclude that there are $2^{\omega(v)}+{2}^{k-\omega(v)}-2$ incident edges to $v$. Hence, the degree of $v$ is given by $2^{\omega(v)}+{2}^{k-\omega(v)}-2$. 
	    
\end{proof}
	
	\begin{remark}
	The degree of a vertex in $3$-dprime, $4$-dprime, and $5$-dprime divisor function graph that were obtained in Example \ref{sampol1} agrees with Theorem \ref{teoverdeg}.  
	\end{remark}
	
	\begin{remark}
	Let $j=0,1,\ldots, k$. In $\Gamma_k$, there are $\binom{k}{j}$ vertices with $j$ distinct prime divisors.  
	\label{remnum}
	\end{remark}
	
	The next result gives a recursive formula in determining the size of $\Gamma_k$. The formula is dependent on the size of $\Gamma_{k-1}$ and the degree of vertices in $\Gamma_{k-1}$. 
	
	\begin{lemma}
	Let $\Gamma_k$ and $\Gamma_{k-1}$ denote the $k$-prime and $k-1$-dprime divisor function graph respectively. If $\Gamma_{k-1}$ has degree sequence $(deg(v_1), deg(v_2),\ldots, deg(v_{2^{k-1}}))$ arranged in increasing order of number of distinct divisors, then
	\begin{equation*}
	    |E(\Gamma_k)|=|E(\Gamma_{k-1})|+\sum_{i=1}^{2^{k-1}}{\left(deg(v_i)+1\right)}.
	\end{equation*}
	\label{lemedge}
	\end{lemma}
	
	\begin{proof}
	
	First, note that $\Gamma_{k-1}$ is a subgraph of $\Gamma_k$. So, all the edges in $\Gamma_{k-1}$ also belong to $\Gamma_k$. Also, observe that $V(\Gamma_k)=V(\Gamma_{k-1})\cup \{vp_k:v\in V(\Gamma_{k-1})\}$. This means that in order to determine the number of edges of $\Gamma_k$, it is enough to consider the number of edges contributed by the vertices in $\{vp_k:v\in V(\Gamma_{k-1})\}$ and add it to $|E(\Gamma_{k-1})|$.    
	
	We claim that if $u\in \{vp_k:v\in V(\Gamma_{k-1})\}$ then $u=vp_k$ contributes $deg(v)+1$ edges in the graph $\Gamma_k$. To prove our claim, we proceed by counting the number of edges contributed by $u$ in $\Gamma_k$ avoiding duplication, which is equal to the number of integers in $V(\Gamma_{k-1})$ that divides $u$ added by the number of integers in $\{vp_k:v\in V(\Gamma_{k-1})\}$ that are divisible by $u$.      
	
	The number of integers in $V(\Gamma_{k-1})$ dividing $u$ is equal to the number of integers in $V(\Gamma_{k-1})$ that divides $v$ plus one (since $v\mid u$). So, we have $2^{\omega(v)}$ integers dividing $u$ in $V(\Gamma_{k-1})$. On the other hand, the number of integers in $\{vp_k:v\in V(\Gamma_{k-1})\}$ that are divisible by $u$ is equal to the number of integers divisible by $v$ in $V(\Gamma_{k-1})$ which is $2^{(k-1)-\omega(v)}-1$. All in all, $u$ contributes a total of $2^{\omega(v)}+2^{(k-1)-\omega(v)}-1=deg(v)+1$.    
	
	Using the just proved claim, we conclude that there are a total of $\displaystyle \sum_{i=1}^{2^{k-1}}{\left(deg(v_i)+1\right)}$ edges contributed by the vertices in the set $\{vp_k:v\in V(\Gamma_{k-1})\}$ in the graph $\Gamma_k$. If we add that sum to $|E(\Gamma_{k-1})|$ we have $|E(\Gamma_k)|$.    
	\end{proof}

A formula on how to compute for $|E(\Gamma_k)|$ using only the variable $k$ is given in the next theorem.

    \begin{theorem}
    The graph $\Gamma_k$ has $3^k-2^k$ number of edges. That is, the size of $\Gamma_k$ is $3^k-2^k$.
    \label{teonumedge}
    \end{theorem}

\begin{proof}
If we combine Lemma \ref{lemedge} with Theorem \ref{teoverdeg} and Remark \ref{remnum} we get 

\begin{equation*}
|E(\Gamma_k)|=|E(\Gamma_{k-1})|+\sum_{j=0}^{k-1}\binom{k-1}{j}(2^{k-j-1}+2^j-1).
\end{equation*}

Now, we wish to simplify the expression $\displaystyle \sum_{j=0}^{k-1}\binom{k-1}{j}(2^{k-j-1}+2^j-1)$ in the above equation by using the identities  $\displaystyle \sum_{j=0}^n{\binom{n}{j}}=2^n$ and  $\displaystyle \sum_{j=0}^n{\binom{n}{j}2^j}=3^n$. By simplifying, we have

\begin{align*}
\sum_{j=0}^{k-1}\binom{k-1}{j}(2^{k-j-1}+2^j-1)&=\sum_{j=0}^{k-1}\binom{k-1}{j}(2^{k-1-j})+\sum_{j=0}^{k-1}\binom{k-1}{j}(2^j)-\sum_{j=0}^{k-1}\binom{k-1}{j}\\
&=3^{k-1}+3^{k-1}-2^{k-1}\\
&=2(3^{k-1})-2^{k-1}.
\end{align*}  
Thus, we now have 

\begin{equation}
|E(\Gamma_k)|=|E(\Gamma_{k-1})|+2(3^{k-1})-2^{k-1}.
\label{eqnE}
\end{equation} 

We note that the $0$-dprime divisor function graph has $|E(\Gamma_0)|=0$. So, solving the recurrence relation in Equation \eqref{eqnE} with the initial condition $|E(\Gamma_0)|=0$ gives

\begin{align*}
|E(\Gamma_k)|&=|E(\Gamma_0)|+2\sum_{j=0}^{k-1}{3^j}-\sum_{j=0}^{k-1}{2^j}\\
&=|E(\Gamma_0)|+2\left(\frac{3^k-1}{2}\right)-(2^k-1)\\
&=0+3^k-1-2^k+1\\
&=3^k-2^k.
\end{align*}

\end{proof}
    
    \begin{remark}
    Using Theorem \ref{teonumedge}, one can verify that $\Gamma_3$ has $3^3-2^3=19$ number of edges. Also, $\Gamma_4$ has $3^4-2^4=65$ number of edges. Finally, $\Gamma_5$ has $3^5-2^5=211$  number of edges. The results agree with the results stated in Example \ref{sampol1}. 
    \end{remark}

We now end the section by presenting some results about distance between vertices in $k$-dprime divisor function graph.

\begin{lemma}
	Let $\Gamma_k$ denote the $k$-dprime divisor function graph. If $u,v\in V(\Gamma_k)$ then
	\begin{equation*}
	d_{\Gamma_k}(u,v)=
	\begin{cases}
	0 &  \mbox{if $u=v$}\\
	1 & \mbox{if $u$ is adjacent to $v$}\\
	2 & \mbox{otherwise.}
	\end{cases}
	\end{equation*}
	\label{lemdis}
	\end{lemma}
	
	\begin{proof}
	    Clearly, $d_{\Gamma_k}(u,v)=0$ if $u=v$ and $d_{\Gamma_k}(u,v)=1$ if $u$ is adjacent to $v$. Now, if $u$ is not adjacent to $v$, then the path $u\to 1\to v$ is a shortest path from $u$ to $v$. Another shortest path from $u$ to $v$ is the path $u\to n\to v$. Hence, $d_{\Gamma_k}(u,v)=2$, if $u$ is not adjacent to $v$.  
	\end{proof}

	\begin{corollary}
	Let $\Gamma_k$ denote the $k$-dprime divisor function graph. The diameter of $\Gamma_k$ denoted by $diam(\Gamma_k)$ is $2$. 
	\end{corollary}

\section{Some Indices of the $k$-dprime Divisor Function Graph}

Given a family of graphs $\cal{G}$, a \textbf{topological index} is a function $Top:\cal{G}\to \mathbb{R}$ such that if $\Gamma_1,\Gamma_2\in \cal{G}$, and $\Gamma_1\cong \Gamma_2$ then $Top(\Gamma_1)=Top(\Gamma_2)$. In this section, we give some general results about the following distance-based and degree-based topological indices of the $k$-dprime divisor function graph $\Gamma_k$

\begin{equation*}
\mbox{\bf Wiener Index:}\ W(\Gamma_k)=\sum_{\{u,v\}\subseteq V(\Gamma_k)}^{}{d_{\Gamma_k}(u,v)}     
\end{equation*}
 
\begin{equation*}
\mbox{\bf Hyper-Wiener Index:}\ WW(\Gamma_k)=\frac{1}{2}\sum_{\{u,v\}\subseteq V(\Gamma_k)}^{}{[d_{\Gamma_k}(u,v)+{(d_{\Gamma_k}(u,v)})^{2}]}  
\end{equation*}

\begin{equation*}
\mbox{\bf Harary Index:}\ H(\Gamma_k)=\sum_{\{u,v\}\subseteq V(\Gamma_k)}^{}{\frac{1}{d_{\Gamma_k}(u,v)}}     
\end{equation*}

\begin{equation*}
\mbox{\bf First Zagreb Index:}\ M_1(\Gamma_k)={M_1}(\Gamma_k)=\sum_{u\in V(\Gamma_k)}^{}{({deg_{\Gamma_k}(u)})^{2}}.
\end{equation*}
 
To effectively calculate the first three indices, we need to recall the concept of graph's \textit{distance matrix} as well as its variants, the \textit{square distance matrix} and the \textit{reciprocal distance matrix}. The \textbf{distance matrix} of a graph $G$ of order $|V(G)|$, denoted by \textbf{D}$(G)$ is the $|V(G)|\times |V(G)|$ matrix \textbf{D} with entries $[d_{ij}]=d_{G}(v_i,v_j)$. On the other hand, the \textbf{square distance matrix} of a graph $G$, denoted by \textbf{D}$^2(G)$ is the matrix with $ij$-entry equal to $(d_{G}(v_i,v_j))^2$. Lastly, the \textbf{reciprocal distance matrix} of a graph $G$, denoted by \textbf{D}$^{-1}(G)$ is the matrix with $ij$-entry equal to $\frac{1}{d_{G}(v_i,v_j)}$. Once the distance matrix of a graph and its variants have been determined, the Wiener, hyper-Wiener, and Harary index of the gaph can be easily calculated as shown in the next example.

\begin{example}
It follows from the graph of the $3$-dprime divisor function graph $\Gamma_3$ in Example \ref{sampol1} that 
\begin{equation*}
\mbox{\textbf{D}}(\Gamma_3)=
\begin{bmatrix}
	0&1&1&1&1&1&1&1 \\
	1&0&2&2&1&1&2&1 \\
	1&2&0&2&1&2&1&1 \\
	1&2&2&0&2&1&1&1 \\
	1&1&1&2&0&2&2&1 \\
	1&1&2&1&2&0&2&1 \\
	1&2&1&1&2&2&0&1 \\
	1&1&1&1&1&1&1&0
	\end{bmatrix}
\end{equation*}
where the matrix is indexed by the ordered set $\{1,p_1,p_2,p_3,p_1p_2,p_1p_3,p_2p_3,p_1p_2p_3\}$. If we use the definition of the Wiener index, we have

\begin{align*}
W(\Gamma_3)&=\sum_{\{u,v\}\subseteq V(\Gamma_3)}^{}{d_{\Gamma_3}(u,v)}\\
&=\sum_{v\in V(\Gamma_3)}^{}{d_{\Gamma_3}(1,v)}+\sum_{v\in V(\Gamma_3)-\{1\}}^{}{d_{\Gamma_3}(p_1,v)}+\ldots +\sum_{v\in V(\Gamma_3)-\{1,p_1,p_2,\ldots, p_2p_3\}}^{}{d_{\Gamma_3}(p_1p_2p_3,v)}\\
&=\frac{1}{2}\sum_{1\leq i,j\leq |V(\Gamma_3)|}{[d_{ij}]}\\
&=\frac{74}{2}\\
&=37.
\end{align*}

In a similar manner, one can show that 

\begin{equation*}
\sum_{\{u,v\}\subseteq V(\Gamma_3)}^{}{(d_{\Gamma_3}(u,v))^2}=\frac{1}{2}\sum_{1\leq i,j\leq |V(\Gamma_3)|}{[d_{ij}]^2}
\end{equation*}

and

\begin{equation*}
\sum_{\{u,v\}\subseteq V(\Gamma_3)}^{}{\frac{1}{d_{\Gamma_3}(u,v)}}=\frac{1}{2}\sum_{1\leq i,j\leq |V(\Gamma_3)|}{\frac{1}{[d_{ij}]}}. 
\end{equation*}

So, by knowing the matrices 

 \begin{equation*}
\mbox{\textbf{D}}^2(\Gamma_3)=
\begin{bmatrix}
	0&1&1&1&1&1&1&1 \\
	1&0&4&4&1&1&4&1 \\
	1&4&0&4&1&4&1&1 \\
	1&4&4&0&4&1&1&1 \\
	1&1&1&4&0&4&4&1 \\
	1&1&4&1&4&0&4&1 \\
	1&4&1&1&4&4&0&1 \\
	1&1&1&1&1&1&1&0
	\end{bmatrix}
\end{equation*}

and 

\begin{equation*}
\mbox{\textbf{D}}^{-1}(\Gamma_3)=
\begin{bmatrix}
	0&1&1&1&1&1&1&1 \\
	1&0&1/2&1/2&1&1&1/2&1 \\
	1&1/2&0&1/2&1&1/2&1&1 \\
	1&1/2&1/2&0&1/2&1&1&1 \\
	1&1&1&1/2&0&1/2&1/2&1 \\
	1&1&1/2&1&1/2&0&1/2&1 \\
	1&1/2&1&1&1/2&1/2&0&1 \\
	1&1&1&1&1&1&1&0
	\end{bmatrix},
\end{equation*}

we can easily compute the hyper-Wiener index and the Harary index of the $3$-dprime divisor function graph as shown in the next page.

\begin{align*}
WW(\Gamma_3)&=\frac{1}{2}\sum_{\{u,v\}\subseteq V(\Gamma_3)}^{}{[d_{\Gamma_3}(u,v)+{(d_{\Gamma_3}(u,v)})^{2}]}\\
&=\frac{1}{2}\left[\sum_{\{u,v\}\subseteq V(\Gamma_3)}^{}{d_{\Gamma_3}(u,v)}+\sum_{\{u,v\}\subseteq V(\Gamma_3)}^{}{(d_{\Gamma_3}(u,v))^2}\right]\\
&=\frac{1}{2}\left[W(\Gamma_3)+\frac{1}{2}\sum_{1\leq i,j\leq |V(\Gamma_3)|}{[d_{ij}]^2}\right]\\
&=\frac{1}{2}(37+55)\\
&=46.
\end{align*} 

\begin{align*}
H(\Gamma_3)&=\sum_{\{u,v\}\subseteq V(\Gamma_3)}^{}{\frac{1}{d_{\Gamma_3}(u,v)}}\\
&=\frac{1}{2}\sum_{1\leq i,j\leq |V(\Gamma_3)|}{\frac{1}{[d_{ij}]}}\\
&=\frac{1}{2}(47)\\
&=23.5.
\end{align*}
\label{sampol2}   
\end{example}

\begin{remark}
In general, given a connected graph $G$, the value of $W(G)$ can be computed by adding all the entries in \textbf{D}$(G)$ and then dividing the result by $2$. For the Harary index, it can be computed by adding all the entries in \textbf{D}$^{-1}(G)$ and then dividing the result by $2$. Finally, for the hyper-Wiener index, it can be calculated by adding half of $W(G)$ to quarter of the sum of all the entries in \textbf{D}$^{2}(G)$.   
\label{remdbind}
\end{remark}

Before we present the general results, we emphasize that the sum of all the vertex degrees in a graph $G$ is equal to $2|E(G)|$. Now, we present the general results.

\begin{theorem}
Let $\Gamma_k$ denote the $k$-dprime divisor function graph. The Wiener index of $\Gamma_k$, denoted by $W(\Gamma_k)$ is given by
	\begin{equation*}
	W(\Gamma_k)=2^{2k}-3^k.
	\end{equation*}
	\label{teoW}
\end{theorem}

\begin{proof}
Let us denote by $S(\mbox{\textbf{D}}(\Gamma_k))$, the sum of all the entries in \textbf{D}$(\Gamma_k)$. Note that the entries  in \textbf{D}$(\Gamma_k)$ are either 0, 1, and 2 as stated in Lemma \ref{lemdis}, and that all in all we have $2^k\times 2^k=2^{2k}$ entries. Clearly, there are $2^k$ 0's in \textbf{D}$(\Gamma_k)$. On the other hand, there are $2(3^k-2^k)$ entries in \textbf{D}$(\Gamma_k)$ whose value is $1$. This is because $d_{\Gamma_k}(u,v)=1$ implies $u$ and $v$ are adjacent, which contributes one count in the vertex degree of $u$ and $v$ respectively. Hence, the total number of $1$'s in \textbf{D}$(\Gamma_k)$ corresponds to the sum of all vertex degrees in $\Gamma_k$ which is equal to $2|E(\Gamma_k)|$. If we apply Theorem \ref{teonumedge} we get the result that there are $2(3^k-2^k)$ entries in \textbf{D}$(\Gamma_k)$ whose value is $1$. Finally, there are $2^{2k}-2^k-2(3^k-2^k)$ entries in \textbf{D}$(\Gamma_k)$ with value $2$.     

Now, by Remark \ref{remdbind} we have

\begin{align*}
W(\Gamma_k)&=\frac{S(\mbox{\textbf{D}}(\Gamma_k))}{2}\\
&=\frac{2^k(0)+2(3^k-2^k)(1)+[2^{2k}-2^k-2(3^k-2^k)](2)}{2}\\
&=(3^k-2^k)+2^{2k}-2^k-2(3^k-2^k)\\
&=2^{2k}-2^k-(3^k-2^k)\\
&=2^{2k}-3^k.
\end{align*}       

\end{proof}

\begin{theorem}
Let $\Gamma_k$ denote the $k$-dprime divisor function graph. The hyper-Wiener index of $\Gamma_k$, denoted by $WW(\Gamma_k)$ is given by
	\begin{equation*}
	WW(\Gamma_k)=2^{k-1}(2^{k+1}+2^k+1)-2(3^k).
	\end{equation*}
	\label{teoWW}
\end{theorem} 

\begin{proof}
If we use proof tecnique similar to the proof of Theorem \ref{teoW} we have

\begin{align*}
WW(\Gamma_k)&=\frac{1}{2}W(\Gamma_k)+\frac{S(\mbox{\textbf{D}}^2(\Gamma_k))}{4}\\
&=\frac{1}{2}W(\Gamma_k)+\left[\frac{2^k(0)+2(3^k-2^k)(1)+[2^{2k}-2^k-2(3^k-2^k)](4)}{4}\right]\\
&=\frac{1}{2}(2^{2k}-3^k)+\left[\frac{2(3^k-2^k)+2^{2k+2}-2^{k+2}-8(3^k-2^k)}{4}\right]\\
&=\frac{2^{2k+1}-2(3^k)+2^{2k+2}-2^{k+2}-6(3^k-2^k)}{4}\\
&=\frac{2^k(2^{k+2}+2^{k+1}+2)-8(3^k)}{4}\\
&=\frac{2^{k+1}(2^{k+1}+2^{k}+1)-8(3^k)}{4}\\
&=2^{k-1}(2^{k+1}+2^k+1)-2(3^k).
\end{align*}
\end{proof}

\begin{theorem}
Let $\Gamma_k$ denote the $k$-dprime divisor function graph. The Harary index of $\Gamma_k$, denoted by $H(\Gamma_k)$ is given by
	\begin{equation*}
	H(\Gamma_k)=\frac{2^{k-1}(2^k-3)+3^k}{2}.
	\end{equation*}
	\label{teoH}
\end{theorem} 

\begin{proof}
If we use proof tecnique similar to the proof of Theorem \ref{teoW} we have

\begin{align*}
H(\Gamma_k)&=\frac{S(\mbox{\textbf{D}}^{-1}(\Gamma_k))}{2}\\
&=\frac{2^k(0)+2(3^k-2^k)(1)+[2^{2k}-2^k-2(3^k-2^k)]\left(\frac{1}{2}\right)}{2}\\
&=\frac{2(3^k-2^k)+2^{2k-1}-2^{k-1}-(3^k-2^k)}{2}\\
&=\frac{2^{2k-1}-2^{k-1}+3^k-2^k}{2}\\
&=\frac{2^{k-1}(2^k-3)+3^k}{2}.
\end{align*}
\end{proof}

\begin{remark}
If we use the results in Theorems \ref{teoW}, \ref{teoWW}, and \ref{teoH} to determine the Wiener, hyper-Wiener, and Harary index of $\Gamma_3$, we get $W(\Gamma_3)=37$, $WW(\Gamma_3)=46$, and $H(\Gamma_3)=23.5$. The computed values agree with those presented in Example \ref{sampol2}.   
\end{remark}

We now end the section by giving the general formula in determining the first Zagreb index of the $k$-dprime divisor function graph.

\begin{theorem}
			 Let $\Gamma_k$ denote the $k$-dprime divisor function graph. The first Zagreb Index of $\Gamma_k$, denoted by $M_1(\Gamma_k)$ is given by 
			 
			 \begin{equation*}
			 M_1(\Gamma_k)=2(2^k-1)^2+\sum_{j=1}^{k-1}{\binom{k}{j}(2^j+2^{k-j}-2)^2}.
			 \end{equation*}
		\end{theorem}
		
		\begin{proof}
			The result follows by combining Theorem \ref{teoverdeg} and Remark \ref{remnum} with the definition of the first Zagreb Index.  
			
		\end{proof}

\section{Other Indices of $k$-dprime Divisor Function Graph}       
	
	The second to the last section of this paper, is dedicated in determining the following topological indices of $k$-dprime divisor function graph for $k=3,4,\mbox{and}\ 5$ using their graph representation given in Example \ref{sampol1}
	
	\begin{equation*}
	\mbox{\bf Second Zagreb Index:}\ {M_2}(G)=\sum_{{uv}\in E(G)}^{}{{deg(u)}{deg(v)}}    
	\end{equation*}
	
	\begin{equation*}
	\mbox{\bf Degree-Distance:}\ DD(G)=\sum_{\{u,v\}\subseteq V(G)}^{}{\bigl[[deg(u)+deg(v)][d(u,v)]\bigr]}
	\end{equation*}
	
	\begin{equation*}
	\mbox{\bf Balaban:}\ J(G)=\frac{m}{\mu+1}\sum_{\{u,v\}\subseteq E(G)}^{}{\bigl({D_u}{D_v}\bigr)}^{-\frac{1}{2}}
	\end{equation*}
	
	\begin{equation*}
	\mbox{\bf Gutman:}\ Gut(G)=\sum_{\{u,v\}\subseteq V(G)}^{}{}{deg(u)deg(v)d(u,v)}
	\end{equation*}
	
	\begin{equation*}
	\mbox{\bf Harmonic:}\ 	Hm(G)=\sum_{\{uv\}\subseteq E(G)}^{}{\frac{2}{deg(u)+deg(v)}}
	\end{equation*}
	
	\begin{equation*}
	\mbox{\bf Randic:}\ R(G)=\sum_{\{uv\}\subseteq E(G)}^{}{\frac{1}{\sqrt{deg(u)deg(v)}}}
	\end{equation*}
	
	\begin{equation*}
	\mbox{\bf First R-index:}\ {R^1}(G)=\sum_{v\in V(G)}^{}{\bigl(r(v)\bigr)^2}
	\end{equation*}
	
	\begin{equation*}
	\mbox{\bf Second R-index:}\ {R^2}(G)=\sum_{uv\in E(G)}^{}{\bigl(r(u)r(v)\bigr)}
	\end{equation*}
	
	\begin{equation*}
	\mbox{\bf Third R-index:}\ {R^3}(G)=\sum_{uv\in E(G)}^{}{\bigl(r(u)+r(v)\bigr)}
	\end{equation*}
	
	\begin{equation*}
	\mbox{\bf Mostar index:}\ {Mo}(G)=\sum_{uv\in E(G)}^{}{\bigl|{n_u}-{n_v}\bigr|}.
	\end{equation*}
	
 \begin{theorem}
If $\Gamma_3$ denote the graph of $3$-dprime divisor function graph, then
			
\begin{equation*}
J(\Gamma_3)=\frac{19}{26}\biggl[\frac{52+12\sqrt{70}}{35}\biggr]
\end{equation*}

\begin{equation*}
DD(\Gamma_3)=338
\end{equation*}
		
\begin{equation*}
Gut(\Gamma_3)=769
\end{equation*}			

\begin{equation*}
Hm(\Gamma_3)=\frac{589}{154}
\end{equation*}						
			
\begin{equation*}
{R^1}(\Gamma_3)=2s^2+6t^2
\end{equation*}			

\begin{equation*}
{R^2}(\Gamma_3)=s^2+12st+15t^2
\end{equation*}			

\begin{equation*}
{R^3}(\Gamma_3)=14s+42t
\end{equation*}			

where $s=7\cdot{4^6}+31$ and $t={7^2}\cdot{4^5}+34$

\begin{equation*}
R(\Gamma_3)=\biggl[\frac{23+12\sqrt{7}}{14}\biggr]
\end{equation*}						
			
\begin{equation*}
{M_2}(\Gamma_3)=481
\end{equation*}			

\begin{equation*}
Mo(\Gamma_3)=36.
\end{equation*}		
\end{theorem}
		
Before we present the proof, let us first consider some definitions.
		
\begin{definition}
			For any simple connected graph $G$ and a vertex $v\in V(G)$, the expressions
			
			\begin{equation*}
			{S_v}=\bigl[\sum_{u\in V(G)}^{}{deg(u)}\bigr]-deg(v)
			\end{equation*}
and
			\begin{equation*}
			{M_v}=\frac{\prod_{u\in V(G)}^{}{deg(u)}}{deg(v)}
			\end{equation*}
are the sum and multiplication degree of $v$, respectively, whereas the $R$ degree of $v$ is defined as $r(v)={S_v}+{M_v}$. Meanwhile, the first $R$ index of $G$ is ${R^1}(G)=\sum_{v\in V(G)}^{}{\bigl(r(v)\bigr)^2}$. Then the second $R$ index of $G$ is ${R^2}(G)=\sum_{uv\in E(G)}^{}{[r(u)r(v)]}$. Finally, the first $R$ index of $G$ is ${R^3}(G)=\sum_{uv\in E(G)}^{}{[r(u)+r(v)]}$.
\end{definition}

\begin{proof}
We will only show the proof for the Randic index, and first, second, and third R indices of $3$-dprime divisor function graph. The other result can be proved similarly.

Based on the definition of the Randic index we have
			\[
			\begin{aligned}
			R(\Gamma_3)&=\sum_{\{uv\}\subseteq E(\Gamma_3)}^{}{\frac{1}{\sqrt{deg(u)deg(v)}}}\\
			&=\frac{1}{2}\biggl[\sum_{\{1v\}\subseteq E(\Gamma_3)}^{}{\frac{1}{\sqrt{deg(1)deg(v)}}}+\sum_{\{{n_1}v\}\subseteq E(\Gamma_3)}^{}{\frac{1}{\sqrt{deg(n_1)deg(v)}}}\\
			&+\sum_{\{{n_2}v\}\subseteq E(\Gamma_3)}^{}{\frac{1}{\sqrt{deg(n_2)deg(v)}}}+\sum_{\{{n_3}v\}\subseteq E(\Gamma_3)}^{}{\frac{1}{\sqrt{deg(n_3)deg(v)}}}\biggr]\\
			&=\frac{1}{2}\biggl[\sum_{\{1v\}\subseteq E(\Gamma_3)}^{}{\frac{1}{\sqrt{7deg(v)}}}+3\Bigl[\sum_{\{{n_1}v\}\subseteq E(\Gamma_3)}^{}{\frac{1}{\sqrt{4deg(v)}}}\Bigr]\\
			&+3\Bigl[\sum_{\{{n_2}v\}\subseteq E(\Gamma_3)}^{}{\frac{1}{\sqrt{4deg(v)}}}\Bigr]+\sum_{\{{n_3}v\}\subseteq E(\Gamma_3)}^{}{\frac{1}{\sqrt{7deg(v)}}}\biggr]\\
			&=\frac{1}{2}\biggl[\frac{1}{\sqrt{7}}\biggl(\frac{1}{\sqrt{7}}+\frac{6}{\sqrt{4}}\biggr)+\frac{3}{\sqrt{4}}\biggl(\frac{2}{\sqrt{7}}+\frac{2}{\sqrt{4}}\biggr)+\frac{3}{\sqrt{4}}\biggl(\frac{2}{\sqrt{7}}+\frac{2}{\sqrt{4}}\biggr)+\frac{1}{\sqrt{7}}\biggl(\frac{1}{\sqrt{7}}+\frac{6}{\sqrt{4}}\biggr)\biggr].
			\end{aligned}
			\]

Simplifying the above equation gives the desired result.

Next, we prove the result on the R indices of $\Gamma_3$. Using the definition presented earlier for the R indices of a graph, the sum and multiplication degree of each vertex in $\Gamma_3$ are
			\[
			\begin{aligned}
			{S_1}&={S_{n_3}}=&38-7=&31\\
			{S_{n_1}}&={S_{n_2}}=&38-4=&34
			\end{aligned}
			\]
			\[
			\begin{aligned}
			{M_1}&={M_{n_3}}=&\frac{{7^2}\cdot{4^6}}{7}=&{7}\cdot{4^6}\\
			{M_{n_1}}&={M_{n_2}}=&\frac{{7^2}\cdot{4^6}}{4}=&{7^2}\cdot{4^5},
			\end{aligned}
			\]
			respectively. Letting $s=r(1)=r(n_3)=7\cdot{4^6}+31$ and $t=r(n_1)=r(n_2)={7^2}\cdot{4^5}+34$, then, by the definition of the first $R$ index, we get
			\[
			\begin{aligned}
			{R^1}(\Gamma_3)&=\bigl(r(1)\bigr)^2+3\biggl[\sum_{v\in V(\Gamma_3)}^{}{\bigl(r(n_1)\bigr)^2}\biggr]+3\biggl[\sum_{v\in V(\Gamma_3)}^{}{\bigl(r(n_2)\bigr)^2}\biggr]+\bigl(r(n_3)\bigr)^2\\
			{R^1}(\Gamma_3)&=s^2+3(t^2)+3(t^2)+s^2\\
			{R^1}(\Gamma_3)&=2s^2+6t^2.
			\end{aligned}
			\]
Using the formula for the second $R$ index,we obtain the following results:
			\[
			\begin{aligned}
			{R^2}(\Gamma_3)&=\frac{1}{2}\biggl[\Bigl[\sum_{1v\in E(\Gamma_3)}^{}{r(1)r(v)}\Bigr]+3\Bigl[\sum_{{n_1}v\in E(\Gamma_3)}^{}{r(n_1)r(v)}\Bigr]\\
			&+3\Bigl[\sum_{{n_2}v\in E(\Gamma_3)}^{}{r(n_2)r(v)}\Bigr]+\Bigl[\sum_{{n_3}v\in E(\Gamma_3)}^{}{r(n_3)r(v)}\bigr]\biggr]\\
			{R^2}(\Gamma_3)&=\frac{1}{2}\bigl[s(s+6t)+3t(2s+5t)+3t(2s+5t)+s(s+6t)\bigr]\\
			{R^2}(\Gamma_3)&=s^2+12st+15t^2.
			\end{aligned}
			\]
Lastly, for the third $R$ index, we have
			\[
			\begin{aligned}
			{R^3}(\Gamma_3)&=\frac{1}{2}\biggl[\Bigl[\sum_{1v\in E(\Gamma_3)}^{}{r(1)+r(v)}\Bigr]+3\Bigl[\sum_{{n_1}v\in E(\Gamma_3)}^{}{r(n_1)+r(v)}\Bigr]\\
			&+3\Bigl[\sum_{{n_2}v\in E(\Gamma_3)}^{}{r(n_2)+r(v)}\Bigr]+\Bigl[\sum_{{n_3}v\in E(\Gamma_3)}^{}{r(n_3)+r(v)}\bigr]\biggr]\\
			{R^3}(\Gamma_3)&=\frac{1}{2}\bigl[[(s+s)+6(s+t)]+3[2(t+s)+5(t+t)]+3[2(t+s)+5(t+t)]+[(s+s)+6(s+t)]\bigr]\\
			{R^3}(\Gamma_3)&=14s+42t.
			\end{aligned}
			\]
\end{proof}
 
\begin{theorem}

If $\Gamma_4$ denote the graph of $4$-dprime divisor function graph, then 

\begin{equation*}
J(\Gamma_4)=\frac{65}{102}\biggl[\frac{202+16\sqrt{330}+66\sqrt{10}+60\sqrt{33}}{165}\biggr]
\end{equation*}

\begin{equation*}
DD(\Gamma_4)=3 712
\end{equation*}
		
\begin{equation*}
Gut(\Gamma_4)=10 557
\end{equation*}			

\begin{equation*}
Hm(\Gamma_4)=\frac{36 367}{4 830}
\end{equation*}						
			
\begin{equation*}
{R^1}(\Gamma_4)=2s^2+8t^2+6w^2
\end{equation*}			

\begin{equation*}
{R^2}(\Gamma_4)=s^2+16st+6sw+15t^2+24tw
\end{equation*}			

\begin{equation*}
{R^3}(\Gamma_4)=24s+70t+30w
\end{equation*}			

where $s=31\cdot {16^{10}}\cdot{10^{20}}+391$, $t={31^2}\cdot {16^9}\cdot{10^{20}}+406$

\begin{equation*}
R(\Gamma_4)=\biggl[\frac{47+60\sqrt{3}+12\sqrt{10}+8\sqrt{30}}{30}\biggr]
\end{equation*}						
			
\begin{equation*}
{M_2}(\Gamma_4)=3993
\end{equation*}			

\begin{equation*}
Mo(\Gamma_4)=268.
\end{equation*}		
						
\end{theorem}
	
\begin{proof}

We will only show the proof for the Gutman index and Harmonic index. The other results can be proved similarly.
			
From the definition of the Harmonic index, and the properties of $\Gamma_4$ we get
\[
\begin{aligned}
Hm(\Gamma_4)&=\sum_{\{uv\}\subseteq E(\Gamma_4)}^{}{\frac{2}{deg(u)+deg(v)}}\\
&=\frac{1}{2}\Biggl[\binom{4}{0}\biggl[\sum_{\{1v\}\subseteq E(\Gamma_4)}^{}{\frac{2}{deg(1)+deg(v)}}\biggr]+\binom{4}{1}\biggl[\sum_{\{{n_1}v\}\subseteq E(\Gamma_4)}^{}{\frac{2}{deg(n_1)+deg(v)}}\biggr]\\
&+\binom{4}{2}\biggl[\sum_{\{{n_2}v\}\subseteq E(\Gamma_4)}^{}{\frac{2}{deg(n_2)+deg(v)}}\biggr]+\binom{4}{3}\biggl[\sum_{\{{n_3}v\}\subseteq E(\Gamma_4)}^{}{\frac{2}{deg(n_3)+deg(v)}}\biggr]\\
&+\binom{4}{4}\biggl[\sum_{\{{n_4}v\}\subseteq E(\Gamma_4)}^{}{\frac{2}{deg(n_4)+deg(v)}}\biggr]\Biggr]\\
&=\frac{1}{2}\Biggl[\biggl[\sum_{\{1v\}\subseteq E(\Gamma_4)}^{}{\frac{2}{15+deg(v)}}\biggr]+4\biggl[\sum_{\{{n_1}v\}\subseteq E(\Gamma_4)}^{}{\frac{2}{8+deg(v)}}\biggr]+6\biggl[\sum_{\{{n_2}v\}\subseteq E(\Gamma_4)}^{}{\frac{2}{6+deg(v)}}\biggr]\\
&+4\biggl[\sum_{\{{n_3}v\}\subseteq E(\Gamma_4)}^{}{\frac{2}{8+deg(v)}}\biggr]+\biggl[\sum_{\{{n_4}v\}\subseteq E(\Gamma_4)}^{}{\frac{2}{15+deg(v)}}\biggr]\Biggr]\\
&=\frac{1}{2}\Biggl[\biggl[\biggl(\frac{2}{15+15}\biggr)(1)+\biggl(\frac{2}{15+8}\biggr)(8)+\biggl(\frac{2}{15+6}\biggr)(6)\biggr]+\biggl[\biggl(\frac{2}{8+15}\biggr)(2)+\biggl(\frac{2}{8+8}\biggr)(3)+\biggl(\frac{2}{8+6}\biggr)(3)\biggr]\\
&+\biggl[\biggl(\frac{2}{6+15}\biggr)(2)+\biggl(\frac{2}{6+8}\biggr)(4)\biggr]+\biggl[\biggl(\frac{2}{8+15}\biggr)(2)+\biggl(\frac{2}{8+8}\biggr)(3)+\biggl(\frac{2}{8+6}\biggr)(3)\biggr]\\
&+\biggl[\biggl(\frac{2}{15+15}\biggr)(1)+\biggl(\frac{2}{15+8}\biggr)(8)+\biggl(\frac{2}{15+6}\biggr)(6)\biggr]\Biggr]\\
&=\frac{36367}{4830}
\end{aligned}
\]
			
Similarly, for the Gutman index, we obtain the following 
\[
\begin{aligned}
Gut(\Gamma_4)&=\sum_{\{u,v\}\subseteq V(G)}^{}{}{deg(u)deg(v)d(u,v)}\\
&=\frac{1}{2}\Biggl[\binom{4}{0}\biggl[\sum_{\{1,v\}\subseteq V(\Gamma_4)}^{}{deg(1)deg(v)d(1,v)}\biggr]+\binom{4}{1}\biggl[\sum_{\{n_1,v\}\subseteq V(\Gamma_4)}^{}{deg(n_1)deg(v)d(n_1,v)}\biggr]\\
&+\binom{4}{2}\biggl[\sum_{\{n_2,v\}\subseteq V(\Gamma_4)}^{}{deg(n_2)deg(v)d(n_2,v)}\biggr]+\binom{4}{3}\biggl[\sum_{\{n_3,v\}\subseteq V(\Gamma_4)}^{}{deg(n_3)deg(v)d(n_3,v)}\biggr]\\
&+\binom{4}{4}\biggl[\sum_{\{n_4,v\}\subseteq V(\Gamma_4)}^{}{deg(n_4)deg(v)d(n_4,v)}\biggr]\Biggr]\\
&=\frac{1}{2}\Biggl[\biggl[\sum_{\{1,v\}\subseteq V(\Gamma_4)}^{}{15deg(v)(1)}\biggr]+4\biggl[\sum_{\{n_1,v\}\subseteq V(\Gamma_4)}^{}{8deg(v)d(n_1,v)}\biggr]\\
&+6\biggl[\sum_{\{n_2,v\}\subseteq V(\Gamma_4)}^{}{6deg(v)d(n_2,v)}\biggr]+4\biggl[\sum_{\{n_3,v\}\subseteq V(\Gamma_4)}^{}{8deg(v)d(n_3,v)}\biggr]\\
&+\biggl[\sum_{\{n_4,v\}\subseteq V(\Gamma_4)}^{}{15deg(v)(1)}\biggr]\Biggr]\\
&=\frac{1}{2}\biggl[15\Bigl[\sum_{\{1,v\}\subseteq V(\Gamma_4)}^{}{deg(v)}\Bigr]+4(8)\Bigl[\sum_{\{{n_1},v\}\subseteq V(\Gamma_4)}^{}{deg(v)d({n_1},v)}\Bigr]+6(6)\Bigl[\sum_{\{{n_2},v\}\subseteq V(\Gamma_4)}^{}{deg(v)d({n_2},v)}\Bigr]\\
&+4(8)\Bigl[\sum_{\{{n_3},v\}\subseteq V(\Gamma_4)}^{}{deg(v)d({n_3},v)}\Bigr]+15\Bigl[\sum_{\{1,v\}\subseteq V(\Gamma_4)}^{}{deg(v)}\Bigr]\biggr].
\end{aligned}
\]

Since the distance between any two distinct vertices in $\Gamma_4$ is just 1 or 2, then
\[
\begin{aligned}
Gut(\Gamma_4)=&\frac{1}{2}\biggl[15\bigl[1(15)(1)+8(8)(1)+6(6)(1)\bigr]+32\bigl[2(15)(1)+3(8)(	1)+4(8)(2)+3(6)(1)+3(6)(2)\bigr]\\
			&+36\bigl[2(15)(1)+4(8)(1)+4(8)(2)+5(6)(2)\bigr]\biggr]+32\bigl[2(15)(1)+3(8)(	1)+4(8)(2)+3(6)(1)+3(6)(2)\bigr]\\
			&+15\bigl[1(15)(1)+8(8)(1)+6(6)(1)\bigr]\\
			=&3712.
			\end{aligned}
\]
\end{proof}

\begin{theorem}
If $\Gamma_5$ denote the graph of $5$-dprime divisor function graph, then 

\begin{equation*}
J(\Gamma_5)=\frac{211}{362}\biggl[\frac{19353+260\sqrt{1426}+920\sqrt{403}+1550\sqrt{598}}{9269}\biggr]
\end{equation*}

\begin{equation*}
DD(\Gamma_5)=19 682
\end{equation*}
		
\begin{equation*}
Gut(\Gamma_5)=124 201
\end{equation*}			

\begin{equation*}
Hm(\Gamma_5)=\frac{45901681}{3106324}
\end{equation*}						
			
\begin{equation*}
{R^1}(\Gamma_5)=2s^2+10t^2+20w^2
\end{equation*}			

\begin{equation*}
{R^2}(\Gamma_5)=s^2+20st+30sw+20t^2+100tw+30w^2
\end{equation*}			

\begin{equation*}
{R^3}(\Gamma_5)=52s+160t+190w
\end{equation*}			

where $s=31\cdot {16^{10}}\cdot{10^{20}}+391$, $t={31^2}\cdot {16^9}\cdot{10^{20}}+406$ and $w={31^2}\cdot {16^10}\cdot{10^19}+412$

\begin{equation*}
R(\Gamma_5)=\biggl[\frac{531+20\sqrt{31}+16\sqrt{310}+310\sqrt{10}}{124}\biggr]
\end{equation*}						
			
\begin{equation*}
{M_2}(\Gamma_4)=47 401
\end{equation*}			

\begin{equation*}
Mo(\Gamma_5)=1 720.
\end{equation*}		
						
\end{theorem}
			
\begin{proof}
We will only show the proof for the Degree-distance Index and second Zagreb Index of $\Gamma_5$. The other indices can be proved similarly.
			
From the definition of the second Zagreb index, we have
\[
\begin{aligned}
{M_2}(\Gamma_5)&=\frac{1}{2}\Biggl[\binom{5}{0}\biggl[\sum_{{1v}\in E(\Gamma_5)}^{}{{deg(1)}{deg(v)}}\biggr]+\binom{5}{1}\biggl[\sum_{{{n_1}v}\in E(\Gamma_5)}^{}{{deg({n_1})}{deg(v)}}\biggr]\\
&+\binom{5}{2}\biggl[\sum_{{{n_2}v}\in E(\Gamma_5)}^{}{{deg({n_2})}{deg(v)}}\biggr]+\binom{5}{3}\biggl[\sum_{{{n_3}v}\in E(\Gamma_5)}^{}{{deg({n_3})}{deg(v)}}\biggr]\\
&+\binom{5}{4}\biggl[\sum_{{{n_4}v}\in E(\Gamma_5)}^{}{{deg({n_4})}{deg(v)}}\biggr]+\binom{5}{5}\biggl[\sum_{{{n_5}v}\in E(\Gamma_5)}^{}{{deg({n_5})}{deg(v)}}\biggr]\Biggr]\Biggr]\\
&=\frac{1}{2}\Biggl[\biggl[\sum_{{1v}\in E(\Gamma_5)}^{}{31deg(v)}\biggr]+5\biggl[\sum_{{{n_1}v}\in E(\Gamma_5)}^{}{16deg(v)}\biggr]+10\biggl[\sum_{{{n_2}v}\in E(\Gamma_5)}^{}{10deg(v)}\biggr]\\
&+10\biggl[\sum_{{{n_3}v}\in E(\Gamma_5)}^{}{10deg(v)}\biggr]+5\biggl[\sum_{{{n_3}v}\in E(\Gamma_5)}^{}{16deg(v)}\biggr]+\biggl[\sum_{{{n_5}v}\in E(\Gamma_5)}^{}{31deg(v)}\biggr]\Biggr]\\
&=\frac{1}{2}\Bigl[\bigl[31+5(16)+10(10)\bigr]+5(16)\bigl[2(31)+4(16)+10(10)\bigr]+10(10)\bigl[2(31)+5(16)+3(10)\bigr]\\
&+10(10)\bigl[2(31)+5(16)+3(10)\bigr]+5(16)\bigl[2(31)+4(16)+10(10)\bigr]+\bigl[31+5(16)+10(10)\bigr]\Bigr]\\
&=47401.
\end{aligned}
\]
			
For the Degree-distance index of$\Gamma_5$, we have
\[
\begin{aligned}
DD(\Gamma_5)&=\frac{1}{2}\Biggl[\binom{5}{0}\biggl[\sum_{\{1,v\}\subseteq V(\Gamma_5)}^{}{[deg(1)+deg(v)]d(1,v)}\biggr]+\binom{5}{1}\biggl[\sum_{\{n_1,v\}\subseteq V(\Gamma_5)}^{}{[deg(n_1)+deg(v)]d(n_1,v)}\biggr]\\
&+\binom{5}{2}\biggl[\sum_{\{n_2,v\}\subseteq V(\Gamma_5)}^{}{[deg(n_2)+deg(v)]d(n_2,v)}\biggr]+\binom{5}{3}\biggl[\sum_{\{n_3,v\}\subseteq V(\Gamma_5)}^{}{[deg(n_3)+deg(v)]d(n_3,v)}\biggr]\\
&+\binom{5}{4}\biggl[\sum_{\{n_4,v\}\subseteq V(\Gamma_5)}^{}{[deg(n_4)+deg(v)]d(n_4,v)}\biggr]+\binom{5}{5}\biggl[\sum_{\{n_5,v\}\subseteq V(\Gamma_5)}^{}{[deg(n_5)+deg(v)]d(n_5,v)}\biggr]\Biggr]\\
&=\frac{1}{2}\Biggl[\biggl[\sum_{\{1,v\}\subseteq V(\Gamma_5)}^{}{[31+deg(v)](1)}\biggr]+5\biggl[\sum_{\{n_1,v\}\subseteq V(\Gamma_5)}^{}{[16+deg(v)]d(n_1,v)}\biggr]\\
&+10\biggl[\sum_{\{n_2,v\}\subseteq V(\Gamma_5)}^{}{[10+deg(v)]d(n_2,v)}\biggr]+10\biggl[\sum_{\{n_3,v\}\subseteq V(\Gamma_5)}^{}{[10+deg(v)]d(n_3,v)}\biggr]\\
&+5\biggl[\sum_{\{n_4,v\}\subseteq V(\Gamma_5)}^{}{[16+deg(v)]d(n_4,v)}\biggr]+\biggl[\sum_{\{n_5,v\}\subseteq V(\Gamma_5)}^{}{[31+deg(v)](1)}\biggr]\Biggr]\\
&=\frac{1}{2}\biggl[\Bigl[(31+31)(1)+10(31+16)(1)+20(31+10)(1)\Bigr]\\
&+5\Bigl[2(16+31)(1)+4(16+16)(1)+5(16+16)(2)+10(16+10)(1)+10(16+10)\Bigr]\\
&+10\Bigl[2(10+31)(1)+5(10+16)(1)+5(10+16)(2)+3(10+10)(1)+16(10+10)\Bigr]\\
&+10\Bigl[2(10+31)(1)+5(10+16)(1)+5(10+16)(2)+3(10+10)(1)+16(10+10)\Bigr]\\
&+5\Bigl[2(16+31)(1)+4(16+16)(1)+5(16+16)(2)+10(16+10)(1)+10(16+10)\Bigr]\\
&+\Bigl[(31+31)(1)+10(31+16)(1)+20(31+10)(1)\Bigr]\biggr]\\
&=62(31)+460(16)+1040(10)\\
&=19 682.
			\end{aligned}
			\]
\end{proof}

\section{Conclusion and Some Problems}
In this paper, we introduced the concept of $k$-dprime divisor function graph and determined some of its basic properties. The general formula for its Wiener, hyper-Wiener, Harary, and First Zagreb index were also presented. We then computed other topological indices of the $k$-dprime divisor function graph for $k=3,4,5$.
	
Since this is an introductory paper about $k$-dprime divisor function graph, there are so many possible problems that the reader might consider. Some possible problems are (1) finding a general closed formula in determining the indices of $k$-dprime divisor function graph that were presented in Section 4, and (2) studying the energy and distance-eigenvalues of the $k$-dprime divisor function graph.
	
\section*{Acknowledgment}
We are thankful to our families and friends for their motivation. Our deepest gratitude also to the Central Luzon State University and the Department of Mathematics and Physics for their unending support. Lastly, we thank Professor Vignesh Ravi from Division of Mathematics, School of Advanced Sciences, Vallore Institute of Technology, Chennai Campus, for clarifying the origin of the  term semiprime divisor function graph.


\begin{thebibliography}{2}

\bibitem{Bon} 
J.A. Bondy, and U.S.R. Murty, {Graph Theory}, Springer, 2008.
	

\bibitem{Rose} 
D.M. Burton, {Elementary Number Theory, Seventh Edition}, The McGraw-Hill Companies, 2010.
			
\bibitem{Chart}
G. Chartrand, R. Muntean, V. Saenpholphat and P. Zhang, \textit{Which graphs are divisor graphs?}, Congr. Numer., 151 (2001), pp. 189--200.
		
\bibitem{Frayer}
C. Frayer, \textit{Properties of Divisor Graphs}, Rose-Hulman Undergraduate Mathematics Journal, 4(2) (2003), pp. 1--10.
	
\bibitem{Kan}
K. Kannan, D. Narasimhan, and S. Shanmugavelan, \textit{The graph of divisor function $D(n)$}, International Journal of Pure and Applied Mathematics, 102(3) (2015), pp. 483--494.

\bibitem{Nara}
D. Narasimhan, A. Elamparithi, and R. Vignesh, \textit{Connectivity, Independency and Colorability of Divisor Function Graph $G_{D(n)}$}, International Journal of Engineering and Advanced Technology, 8(2S) (2018), pp. 209--213.

\bibitem{Shan}
S. Shanmugavelan, K.T. Rajeswari, and C. Natarajan, \textit{A note on indices of primepower and semiprime divisor function graph}, TWMS J. App. and Eng. Math., 11(special issue) (2021), pp. 51--62.
		
\bibitem{Singh} 
G.S. Singh, and G. Santhosh, {\em Divisor graphs - I}, Preprint.

\bibitem{Tsao} 
Y.-P. Tsao, {\em A simple research of divisor graphs}, The 29th Workshop on Combinatorial Mathematics and Computation Theory.
	
\bibitem{Vinh} 
L.A. Vinh, {\em Divisor graphs have arbitrary order and size}, AWOCA (2006).
		
		
	

\end{thebibliography}
\end{document}